\documentclass{amsart}
\usepackage{anysize}
\usepackage[utf8x]{inputenc}
\usepackage[english]{babel}
\usepackage{amssymb, amsmath, amsthm}
\usepackage{enumitem}
\usepackage{mathtools}
\usepackage{url}
\usepackage[all,cmtip]{xy}
\usepackage[hidelinks,bookmarksnumbered]{hyperref}

\let\cal\mathcal
\def\AA{{\cal A}}

\def\CC{{\cal C}}
\def\DD{{\cal D}}
\def\EE{{\cal E}}
\def\FF{{\cal F}}

\def\VV{{\cal V}}

\let\blb\mathbb

\def\bZ{{\blb Z}}

\def\bR{{\blb R}}

\def\bZ{{\blb Z}}

\author{Oliver Braunling}
\address{Oliver Braunling \\ Albert-Ludwigs-University Freiburg \\ Institute for Mathematics \\ D-79104 Freiburg \\ Germany}
\email{oliver.braeunling@math.uni-freiburg.de}
\author{Ruben Henrard}
\address{Ruben Henrard \\ Universiteit Hasselt \\ Campus Diepenbeek \\ Departement WNI \\ 3590 Diepenbeek \\ Belgium}
\email{ruben.henrard@uhasselt.be}
\author{Adam-Christiaan van Roosmalen}
\address{Adam-Christiaan van Roosmalen \\ Universiteit Hasselt \\ Campus Diepenbeek \\ Departement WNI \\ 3590 Diepenbeek \\ Belgium}
\email{adamchristiaan.vanroosmalen@uhasselt.be}

\title[$K$-Theory of Locally Compact Modules]{$K$-Theory of Locally Compact Modules over Orders}
\keywords{Locally compact modules, $K$-theory, exact category}
\subjclass[2020]{19B28, 19F05, 22B05; 18E35}

\date{\today}
\thanks{The first author was supported by DFG GK1821 \textquotedblleft Cohomological Methods
in Geometry\textquotedblright.}
\thanks{The third author was supported by FWO (12.M33.16N)}

\newtheorem{theorem}{Theorem}[section]
\newtheorem{proposition}[theorem]{Proposition}
\newtheorem{lemma}[theorem]{Lemma}
\newtheorem{corollary}[theorem]{Corollary}

\theoremstyle{definition}
\newtheorem{definition}[theorem]{Definition}
\newtheorem{remark}[theorem]{Remark}

\newtheorem{construction}[theorem]{Construction}

\DeclareMathOperator{\inflation}{\rightarrowtail}
\DeclareMathOperator{\deflation}{\twoheadrightarrow}

\newcommand{\ex}[1]{#1^{\operatorname{ex}}}

\DeclareMathOperator{\im}{im}
\DeclareMathOperator{\coker}{coker}

\DeclareMathOperator{\Hom}{Hom}
\DeclareMathOperator{\Fun}{Fun}

\DeclareMathOperator{\Ext}{Ext}

\DeclareMathOperator{\Mod}{Mod}
\DeclareMathOperator{\smod}{mod}
\renewcommand{\mod}{\operatorname{mod}}
\DeclareMathOperator{\fmod}{fmod}

\DeclareMathOperator{\LCA}{\mathsf{LCA}}
\DeclareMathOperator{\oLCA}{\overline{\mathsf{LCA}}}

\DeclareMathOperator{\LCAA}{\mathsf{LCA}_{\mathfrak{A}}}

\DeclareMathOperator{\LCAAC}{\mathsf{LCA}_{\mathfrak{A},\mathsf{C}}}
\DeclareMathOperator{\LCAAD}{\mathsf{LCA}_{\mathfrak{A},\mathsf{D}}}
\DeclareMathOperator{\LCAAR}{\mathsf{LCA}_{\mathfrak{A},\mathbb{R}}}

\DeclareMathOperator{\LCAACV}{\mathsf{LCA}_{\mathfrak{A},\mathsf{C}\mathbb{R}}}
\DeclareMathOperator{\LCAAf}{\mathsf{LCA}_{\mathfrak{A},\mathrm{f}}}

\DeclareMathOperator{\oLCAA}{\oLCA_\mathfrak{A}}

\DeclareMathOperator{\oLCAAD}{\oLCA_{\mathfrak{A},\mathsf{D}}}

\DeclareMathOperator{\oLCAACV}{\oLCA_{\mathfrak{A},\mathsf{C}\mathbb{R}}}

\DeclareMathOperator{\Ac}{\textbf{Ac}}
\DeclareMathOperator{\Kb}{\mathsf{K}^{\mathsf{b}}}

\DeclareMathOperator{\Db}{\mathsf{D}^{\mathsf{b}}}
\DeclareMathOperator{\DAb}{\mathsf{D}^{\mathsf{b}}_{\AA}}

\DeclareMathOperator{\Dbinf}{\mathsf{D}^{\mathsf{b}}_\infty}

\newcommand{\StabInfCat}{\mathrm{Cat}^{\mathrm{Ex}}_\infty}

\newcommand{\QC}{Q_\mathsf{C}}
\newcommand{\QCV}{Q_{\mathsf{C}\mathbb{R}}}
\newcommand{\QV}{Q_{\mathbb{R}}}
\newcommand{\QA}{Q_\A}
\newcommand{\QfA}{Q_{\A,f}}
\newcommand{\Qf}{Q_f}

\newcommand{\A}{\mathfrak{A}}

\begin{document}

\begin{abstract}
We present a quick approach to computing the $K$-theory of the category of locally compact modules over any order in a semisimple $\mathbb{Q}$-algebra.  We obtain the $K$-theory by first quotienting out the compact modules and subsequently the vector modules.  Our proof exploits the fact that the pair (vector modules plus compact modules, discrete modules) becomes a torsion theory after we quotient out the finite modules.  Treating these quotients as exact categories is possible due to a recent localization formalism.
\end{abstract}

\maketitle

\tableofcontents
\section{Introduction}
Suppose $A$ is a finite-dimensional semisimple $\mathbb{Q}$-algebra and
$\mathfrak{A}\subset A$ is any $\mathbb{Z}$-order.  We write $\smod(-)$ for the category of finitely generated right modules, $A_{\mathbb{R}}$ for $A \otimes_{\mathbb{Q}} \mathbb{R}$, and $\LCAA$ for the exact category of locally compact topological $\mathfrak{A}$-modules \cite{MR2329311}.  We give a new proof for the following theorem.

\begin{theorem}\label{theorem:MainTheorem}
For every localizing invariant $K\colon \StabInfCat \to \DD$ (where $\DD$ is any stable $\infty$-category), there is a canonical fiber sequence:
	\[K(\smod(\mathfrak{A}))\to {K}(\smod(A_{\mathbb{R}}))\to {K}(\LCAA),\]
where the first map is induced by the natural embedding $- \otimes_\bZ \bR \colon\A \to \A_\bR=A_\bR.$
\end{theorem}

For localizing invariants and $\StabInfCat$ we refer to the framework and notation of \cite{MR3070515}.  The principal example is non-connective $K$-theory taking values in spectra (and we will indicatively always denote the invariant by $K$ throughout the paper).  When convenient, and for example in the above statement, we write $K(\mathsf{C})$ even if $\mathsf{C}$ is an exact (or one-sided exact) category, for $K(\Dbinf(\mathsf{C}))$, where $\Dbinf(\mathsf{C})$ is the stable $\infty$-category of bounded complexes attached to $\mathsf{C}$.

The first theorem of the above kind is due to Clausen \cite[Theorem 3.4]{clausen}, who proved it in the special case $A=\mathbb{Q}$ and $\mathfrak{A}=\mathbb{Z}$ (with an additional, but ultimately inconsequential, restriction to second-countable topologies) with an eye to applications in class field theory. To this end, he set up a cone construction on the level of stable $\infty$-categories. The above version stems from \cite[Theorem 11.4]{Braunling18}. It was based on Schlichting's Localization Theorem from \cite{Schlichting04}.  In our new approach, we use the recent Localization Theorem of \cite{HenrardVanRoosmalen19b, HenrardVanRoosmalen19a}, which uses the additional flexibility of one-sided exact categories. These devices can be thought of as convenient tools to avoid having to handle the underlying stable $\infty$-categories (or triangulated categories) manually.

The proof of the main theorem is given in \S\ref{section:ProofOfMainTheorem}.  We start by considering the quotient of $\LCAA$ by the subcategory of compact modules $\LCAAC.$  While this subcategory does not satisfy the $s$-filtering conditions of Schlichting's Localization Theorem, it does satisfy the conditions of the recent Localization Theorem of \cite{HenrardVanRoosmalen19b, HenrardVanRoosmalen19a}.  The latter theory endows the quotient $\EE \coloneqq \LCAA / \LCAAC$ with the structure of a one-sided exact category (in the sense of \cite{BazzoniCrivei13, Rump11}), which can then canonically be embedded in its exact hull $\ex{\EE}$.  It is from this exact hull that we take a further quotient, this time by the subcategory $\VV$ of vector modules of $\LCAA$.  Finally, we show that the resulting category $\FF\coloneqq \ex{\EE}/ \VV$ is equivalent to $\Mod \mathfrak{A} / \mod \mathfrak{A}$; it is from this equivalence that we obtain the sequence in Theorem \ref{theorem:MainTheorem}.

The equivalence $\FF \simeq \Mod \mathfrak{A} / \mod \mathfrak{A}$ is based on the universal properties of the aforementioned quotients and the exact hull, as well as on the following observation (see Theorem \ref{theorem:TorsionPair}): after quotienting the finite modules out, the Structure Theorem of $\LCAA$ (Theorem \ref{theorem:StructureTheoremOfLocallyCompactModules}) implies that the pair $(\LCAACV, \LCAAD)$ becomes a torsion pair.  This means that the sequences ``$M_{\textrm{compact}} \oplus M_{\textrm{vector}} \inflation M \deflation M_{\textrm{discrete}}$'' given by the Structure Theorem become essentially unique in the quotient.

Using these new tools, our proof of Theorem \ref{theorem:MainTheorem} is considerably shorter and technically less involved.  We never seriously leave the world of $1$-categories, do not use the $\infty$-categorical cone construction of \cite{clausen}, nor do we need the tedious verifications of Schlichting's left/right $s$-filtering conditions done in \cite{Braunling18}.
\section{Localizations of exact categories}\label{section:Preliminaries}

This section is preliminary in nature.  We summarize the results of \cite{HenrardVanRoosmalen19b,HenrardVanRoosmalen19a} about localizations of exact categories.  One salient feature of the localizations we consider is that the resulting category need not be exact, but will be one-sided exact (in the sense of \cite{BazzoniCrivei13, Rump11}).

\subsection{One-sided exact categories}

\begin{definition}
	A \emph{conflation category} is an additive category $\CC$ together with a chosen class of kernel-cokernel pairs (closed under isomorphisms), called \emph{conflations}. The kernel part of a conflation is called an \emph{inflation} and the cokernel part of a conflation is called a \emph{deflation}. We depict inflations by $\inflation$ and deflations by $\deflation$.
	
	An additive functor $F\colon \CC\to \DD$ between conflation categories is called \emph{conflation-exact} if conflations are mapped to conflations.
\end{definition}

\begin{definition}\label{Definition:OneSidedExactCategory}
	A conflation category $\EE$ is called an \emph{inflation-exact} or \emph{left exact category} if $\EE$ satisfies the following axioms:
	\begin{enumerate}[label=\textbf{L\arabic*},start=0]
		\item\label{L0} For each $X\in \EE$, the map $0\to X$ is an inflation.
		\item\label{L1} The composition of two inflations is again an inflation.
		\item\label{L2} The pushout of any morphism along an inflation exists, moreover, inflations are stable under pushouts.
	\end{enumerate}
	Dualizing the above axioms yields the notion of a \emph{deflation-exact} category. A \emph{Quillen exact} category is simply a conflation category which is both inflation-exact and deflation-exact by \cite[Appendix~A]{Keller90}.
\end{definition}

Let $\EE$ be one-sided exact category. Analogous to exact categories, one can define the bounded derived category $\Db(\EE)$ as the Verdier localization $\Kb(\EE)/\left\langle\Ac(\EE)\right\rangle_{\text{thick}}$ of the bounded homotopy category by the thick closure of the triangulated subcategory of acyclic complexes (see \cite[Corollary~7.3]{BazzoniCrivei13}). The canonical embedding $i\colon\EE\to\Db(\EE)$, mapping objects to stalk complexes in degree zero, is a fully faithful embedding mapping conflations to triangles.
We write $\Dbinf(\EE)$ for the corresponding stable $\infty$-category; in particular, the homotopy category of $\Dbinf(\EE)$ recovers $\Db(\EE)$.

The derived category allows a construction of the \emph{exact hull} $\ex{\EE}$ of $\EE$ (see \cite{HenrardVanRoosmalen19b}): the exact hull is given by the extension closure of $\EE$ in the (bounded) derived category $\Db(\EE)$.  A sequence $X\to Y \to Z$ in $\ex{\EE}$ is a conflation if and only if it fits in a triangle $X \to Y \to Z\to \Sigma X.$  The following proposition is shown in \cite{HenrardVanRoosmalen19b}.

\begin{proposition}\label{proposition:ExactHull}
Let $j\colon \EE \to \ex{\EE}$ be the embedding of an inflation-exact category in its exact hull.
\begin{enumerate}
	\item The embedding $j\colon \EE\to \ex{\EE}$ is $2$-universal among conflation-exact functors to exact categories.
	\item The embedding $j$ lifts to an equivalence $\Dbinf(\EE)\stackrel{\simeq}{\rightarrow}\Dbinf(\ex{\EE})$ of stable $\infty$-categories.
\end{enumerate}
\end{proposition}

\subsection{Strictly percolating subcategories}

\begin{definition}\label{Definition:InflationPercolating}
	Let $\EE$ be an exact category. A full subcategory $\AA\subseteq \EE$ is called a \emph{strictly inflation-percolating subcategory} if the following properties are satisfied:
\begin{enumerate}[label=\textbf{A\arabic*},start=1]
		\item\label{A1} The category $\AA$ is a \emph{Serre subcategory} of $\EE$, i.e.~for any conflation $X\inflation Y\deflation Z$ in $\EE$, we have $Y\in \AA\Leftrightarrow X,Z\in \AA$. 
		\item\label{A2} Every morphism $f\colon A\to X$ with $A\in \AA$ is \emph{strict}, i.e.~factors as $A\deflation \im(f)\inflation X$, and $\im(f)\in \AA$.
	\end{enumerate}
\end{definition}

\begin{remark}
If $\AA\subseteq \EE$ is a strictly inflation-percolating subcategory, then $\AA$ is a fully exact abelian subcategory of $\EE$.
\end{remark}

The following observation will be of use later.

\begin{proposition}\label{proposition:A1A2+InjectiveLiftsToHull}
Let $\EE$ be an inflation-exact category and let $\VV\subseteq \EE$ be a full additive subcategory satisfying axioms \ref{A1} and \ref{A2}. If every object of $\VV$ is injective in $\EE$, then $\VV$ is a strictly inflation-percolating subcategory of the exact hull $\ex{\EE}$.
\end{proposition}

\begin{proof}
As each $V\in \VV\subseteq \EE$ is injective, $\Hom_{\EE}(-,V)$ is exact and hence $\Ext_\EE^1(-,V)= 0.$  As $\ex{\EE}$ is the extension-closure of $\EE$ in $\ex{\EE}$, it follows that $V$ is injective in $\ex{\EE}$ as well.
	
	Note that $\ex{\EE}=\bigcup_{n\geq 0}\EE_n$ where $\EE_0=\EE$ and $\EE_n$ for $n\geq 1$ is defined recursively as extensions of objects in $\EE_{n-1}$. As $\VV\subseteq \ex{\EE}$ consists of injective objects, it follows that $\VV\subseteq \ex{\EE}$ is an extension-closed subcategory.
	
	We now show axiom \ref{A2}. Let $f\colon V\to Y$ be a map in $\ex{\EE}$ with $V\in \VV$. By definition, there is an $n$ such that $Y\in \EE_{n}$. We proceed by induction on $n\geq 0$. If $n=0$, the result follows as $\VV\subseteq \EE$ satisfies axiom \ref{A2}. If $n\geq 1$, then there is a conflation $X\stackrel{\iota}{\inflation} Y\stackrel{\rho}{\deflation} Z$ in $\ex{\EE}$ with $X,Z\in \EE_{n-1}$. By the induction hypothesis, the composition $\rho\circ f$ factors as $V\stackrel{p}{\deflation}V''\stackrel{h}{\inflation}Z$ with $V''\in \VV$. As $\VV\subseteq\ex{\EE}$ is an abelian subcategory, $(V' \coloneqq) \ker(p)\in \VV$. Note that there is an induced map $g\colon V'\to X$ such that $fi=\iota g$. Again the induction hypothesis yields that $g$ factors as $V'\stackrel{g'}{\deflation}U\stackrel{g''}{\inflation}X$ with $U\in \VV$. Taking the pushout of $g'$ along $i$ in $\ex{\EE}$ yields the following commutative diagram (where the rows are conflations):
	\[\xymatrix{
		V'\ar@{>->}[r]^i\ar@{->>}[d]^{g'} & V\ar@{->>}[r]^{p}\ar@{->>}[d]^{f'} & V''\ar@{=}[d]\\
		U\ar@{>->}[r]^{i'}\ar@{>->}[d]^{g''} & W\ar@{->>}[r]^{p'}\ar[d]^{f''} & V''\ar@{>->}[d]^h\\
		X\ar@{>->}[r]^{\iota} & Y\ar@{->>}[r]^{\rho} & Z
	}\] Here the upper-left square is bicartesian, $f''f'=f$ and $W\in \VV$ as $\VV\subseteq\ex{\EE}$ is extension-closed. It follows from \cite[Corollary~3.2]{Buhler10} that $f''$ is an inflation in $\ex{\EE}$. This shows axiom \ref{A2}.
	
	To show axiom \ref{A1}, it remains to show that given a conflation $X\stackrel{\iota}{\inflation}V \stackrel{\rho}{\deflation} Z$ in $\ex{\EE}$ with $V\in\VV$, $X,Z$ belong to $\VV$ as well. By axiom \ref{A2}, $\rho$ is admissible with image in $\VV$.  It follows that $Z\in \VV$. As $X$ is the kernel of a morphism in $\VV$ and $\VV$ is an abelian subcategory of $\EE$, we know that $X$ belongs to $\VV$ as well. This concludes the proof.
\end{proof}

\subsection{Quotients by strictly percolating subcategories}

The next definition is based on \cite[Definition~1.12]{Schlichting04}.

\begin{definition}
	Let $\EE$ be an inflation-exact category and let $\AA\subseteq\EE$ be a strictly inflation-percolating subcategory. A morphism $f\colon X\to Y$ in $\EE$ is called a \emph{weak $\AA$-isomorphism} (or simply a \emph{weak isomorphism}) if $f$ is strict and $\ker(f),\coker(f)\in \AA$. The set of weak $\AA$ -isomorphisms is denoted by $S_{\AA}$.
\end{definition}

The following theorem summarizes the main results of \cite{HenrardVanRoosmalen19b,HenrardVanRoosmalen19a}.

\begin{theorem}\label{theorem:Quotient/LocalizationTheorem}
	Let $\AA$ be a strictly inflation-percolating subcategory of an exact category $\EE$.
	\begin{enumerate}
		\item The set $S_{\AA}$ of weak $\AA$-isomorphisms is a left multiplicative system.
		\item The natural localization functor $Q\colon\EE\to \EE[S^{-1}_\AA]$ endows the localization $\EE[S^{-1}_\AA]$ with an inflation-exact structure such that $Q$ preserves and reflects conflations.
		\item  The localization functor $Q$ is also a quotient in the category of conflation categories, i.e.~it satisfies the following universal property: if $F\colon \EE\to \FF$ is a conflation-exact functor between conflation categories such that $F(\AA)=0$, then $F$ factors uniquely through $Q$ via a conflation-exact functor $F'\colon \EE/\AA=\EE[S_\AA^{-1}]\to \FF$.
	\end{enumerate}
	
	Moreover, the localization sequence $\AA\hookrightarrow\EE \stackrel{Q}{\rightarrow}\EE/\AA$ induces a Verdier localization sequence on the bounded derived categories
	\[\DAb(\EE)\to \Db(\EE)\to \Db(\EE/\AA)\]
	where $\DAb(\EE)$ is the thick subcategory of $\Db(\EE)$ generated by $\AA$ under the canonical embedding $\EE\hookrightarrow \Db(\EE)$. 
	
	If $\AA$ has enough $\EE$-injectives, then the natural embedding $\Db(\AA)\hookrightarrow \DAb(\EE)$ is a triangle equivalence, and there is an exact sequence in $\StabInfCat$:
	\[\Dbinf(\AA)\to \Dbinf(\EE)\to \Dbinf(\EE/\AA).\]
\end{theorem}
\section{Structure theory of locally compact modules over an order}

Let $\LCA$ be the exact category of locally compact abelian groups, cf.~\cite{MR2329311}. Let $A$ denote a finite-dimensional semisimple $\mathbb{Q}$-algebra and $\mathfrak{A}\subset A$ is a $\mathbb{Z}$-order, i.e.~a subring of $A$ which is a finitely generated $\mathbb{Z}$-module such that $\mathbb{Q}\cdot \mathfrak{A}=A$. In this section, we have a closer look at the category $\LCAA$ of locally compact right $\A$-modules.

\begin{definition}\label{Definition:LocallyCompactModulesOverOrder}
	We define the category $\LCAA$ of locally compact right modules over $\mathfrak{A}$ as follows:
		\begin{enumerate}
			\item An object $M\in \LCAA$ is a right $\mathfrak{A}$-module such that the additive group $(M,+)\in \LCA$ is a locally compact group and such that right multiplication by any $\alpha\in M$ is a continuous endomorphism of $(M,+)$.
			\item A morphism $f\colon M\to N$ is a continuous right $\mathfrak{A}$-module map.
		\end{enumerate}
\end{definition}

The following theorem is standard, see \cite[Proposition~3.4 and Lemma~3.6]{Braunling18} (based on the earlier \cite{MR2329311}).

\begin{theorem}
The category $\LCAA$ is a quasi-abelian category; the inflations are given by closed injections and deflations are given by open surjections.
\end{theorem}

\begin{definition}
	We consider the following subcategories of $\LCAA$:
	\begin{enumerate}
		\item $\LCAAC$ denotes the full subcategory of compact $\mathfrak{A}$-modules.
		\item $\LCAAD$ denotes the full subcategory of discrete $\mathfrak{A}$-modules.
		\item $\LCAAR$ denotes the full subcategory of vector $\mathfrak{A}$-modules, i.e.~those $\mathfrak{A}$-modules whose underlying locally compact abelian group is isomorphic to $\mathbb{R}^n$ for some finite $n\geq 0$.
		\item $\LCAACV$ denotes the full subcategory of $\A$-modules which are a direct sum of a compact and a vector $\A$-module. 
		\item $\LCAAf$ denotes the full subcategory of finite $\A$-modules.
	\end{enumerate}
\end{definition}

The Structure Theorem for locally compact abelian groups extends to $\mathfrak{A}$-modules in the following sense (see \cite[Lemma~6.5]{Braunling18}).

\begin{theorem}[Structure Theorem for locally compact modules]\label{theorem:StructureTheoremOfLocallyCompactModules}
For each $M \in \LCAA$, there exists a (non-canonical) conflation
\[C_M\oplus V_M\stackrel{i_M}{\inflation} M \stackrel{p_M}{\deflation} D_M\]
	with $C_M\in \LCAAC, D_M\in \LCAAD$ and $V_M\in \LCAAR$.
\end{theorem}

It is well known that vector groups are both injective and projective in $\LCA$ (see for example \cite[Corollary~3 to Theorem~3.3]{Moskowitz67}). This result extends to $\mathfrak{A}$-modules (see \cite[Theorem~5.13]{Braunling18}).

\begin{theorem}\label{theorem:VectorModulesAreBothInjectiveAndProjectiveInLCAA}
	The vector $\mathfrak{A}$-modules are both injective and projective in $\LCAA$.
\end{theorem}

The next lemma will be useful later.

\begin{lemma}\label{lemma:QuotientsAndDiscreteSubobjectsOfVectorModules}
	Let $X \in \LCAACV$, thus $X \cong C \oplus V$ with $C \in \LCAAC$ and $V \in \LCAAR.$
	\begin{enumerate}
	  \item If $X \inflation Y \deflation Z$ is a conflation with $Z \in \LCAAC$, then $Y \in \LCAACV.$
		\item If $f\colon X\deflation Y$ is a deflation, then $Y\in \LCAACV$.
		\item	If $g\colon Y\inflation X$ is an inflation such that $Y\in \LCAAD$, then $Y$ is a finitely generated $\mathfrak{A}$-module. 
	\end{enumerate}
\end{lemma}

\begin{proof}
\begin{enumerate}
	\item As $V$ is injective, the conflation $X \inflation Y \deflation Z$ is a direct sum of conflations $V \inflation V \deflation 0$ and $C \inflation C' \deflation Z$.  Since $\LCAAC$ is closed under extensions, we find that $C'\in \LCAAC$, as required.
	\item Applying the Structure Theorem to $Y$ yields a conflation $C' \oplus V' \inflation Y \deflation D$, with $D \in \LCAAD.$  From the previous statement, we see that it suffices to show that $D \in \LCAAC$ (thus, $D$ is finite).  Write $h$ for the composition $X \deflation Y \deflation D$ and consider the conflation $\ker h \inflation X \deflation D.$  As $X \cong C \oplus V$ and $\Hom(V,D) = 0$, we see that this conflation is the direct sum of conflations $V \inflation V \deflation 0$ and $K \inflation C \deflation D.$  As $C$ is compact, we find that $D$ is compact as well. %
	\item It suffices to show that $Y$ is finitely generated as an abelian group.  As $Y$ is a closed subgroup of $X$, the Pontryagin dual of \cite[Chapter 2, Corollary~2 to Theorem 7]{Morris77} implies that $Y$ is of the form $\mathbb{R}^m\oplus \mathbb{Z}^l\oplus D$ with $D$ a finite (discrete) group. As $Y$ is discrete by assumption, we see that $Y$ is a finitely generated group.
		\qedhere
\end{enumerate}
\end{proof}

We now interpret these results in the quotient category $\oLCAA \coloneqq \LCAA / \LCAAf$.  We write $\oLCAAD$ for the full subcategory of $\oLCAA$ consisting of those objects which are discrete groups; the subcategory $\oLCAACV$ of $\oLCAA$ is defined similarly.

\begin{proposition}\label{proposition:fLocalization}
The category $\LCAAf$ is a strictly two-sided percolating subcategory of $\LCAA.$  Moreover, the following hold:
\begin{enumerate}
	\item\label{enumerate:fLocalization1} the quotient $\LCAA / \LCAAf$ is quasi-abelian,
	\item\label{enumerate:fLocalization2} the localization $\Qf\colon \LCAA \to \LCAA / \LCAAf$ commutes with finite limits and colimits,
	\item\label{enumerate:fLocalization3} the subcategories $\oLCAAD$ and $\oLCAACV$ of $\oLCAA$ are closed under isomorphisms,
	\item\label{enumerate:fLocalization4} $\oLCAAD = \LCAAD / \LCAAf$ and $\oLCAACV = \LCAACV / \LCAAf,$
	\item\label{enumerate:fLocalization5} any morphism $X \to Y$ in $\oLCAAD$ is strict.
\end{enumerate}
\end{proposition}

\begin{proof}
It is easy to verify that $\LCAAf$ is a strictly inflation- and deflation-percolating in $\LCAA.$  It follows from \cite{HenrardVanRoosmalen19a} that $\LCAA / \LCAAf$ is quasi-abelian.  As the set $S_{\LCAAf}$ is a left and a right multiplicative set, the localization $\Qf$ commutes with finite limits and colimits.  For \eqref{enumerate:fLocalization3}, recall from Theorem \ref{theorem:Quotient/LocalizationTheorem} that $S_{\LCAAf}$ is saturated.  So, we can reduce to showing that for any weak isomorphism $s\colon X \stackrel{\sim}{\rightarrow} Y$, we have that $X \in \LCAAD$ (or in $\LCAACV$) if and only if $Y\in \LCAAD$ (or in $\LCAACV$).  As every weak isomorphism is a composition of inflations and deflations in (with cokernel and kernel in $\LCAAf$), we can furthermore assume that $s$ is of this form.  These cases are then easily handled separately.

The last two statements follow easily from \eqref{enumerate:fLocalization3}. 
\end{proof}

\begin{theorem}[Structure Theorem for $\oLCAA$]\label{theorem:TorsionPair}
The pair $(\oLCAACV, \oLCAAD)$ is a torsion pair in $\oLCAA$, meaning that $\Hom(\oLCAACV, \oLCAAD) = 0$ and every $M \in \oLCAA$ fits in a conflation
\[C_M \stackrel{i_M}{\inflation} M \stackrel{p_M}{\deflation} D_M\]
with $C_M \in \oLCAACV$ and $D_M \in \oLCAAD.$  Such a conflation is then unique up to unique isomorphism.
\end{theorem}

\begin{proof}
Directly from Theorem \ref{theorem:StructureTheoremOfLocallyCompactModules}.
\end{proof}

\begin{corollary}\label{corollary:DiscreteInCV}
Let $X \stackrel{f}{\inflation} Y \stackrel{g}{\deflation} Z$ be a conflation in $\oLCAA$.
\begin{enumerate}
	\item If $Y \in \oLCAACV$, then $Z \in \oLCAACV$.
	\item If $Y \in \oLCAAD$, then $X, Z \in \oLCAAD.$
	\item If $Y \in \oLCAACV$ and $X\in \oLCAAD$, then $X$ is finitely generated.
\end{enumerate}
\end{corollary}

\begin{proof}
We only prove the first statement.  The other statements can be proven in an analogous way.  Let $X \stackrel{\sim}{\leftarrow} X' \stackrel{f'}{\rightarrow} Y$ be a roof in $\LCAA$ representing $f$; in $\oLCAA$, we have $\coker f \cong \coker f'.$  It follows from Lemma \ref{lemma:QuotientsAndDiscreteSubobjectsOfVectorModules} that $\coker f'\in \LCAACV$ and hence $\coker f \in \LCAACV$ by Proposition \ref{proposition:fLocalization}.(\ref{enumerate:fLocalization3}).
\end{proof}
\section{\texorpdfstring{$K$}{K}-theory of locally compact modules over an order}\label{section:ProofOfMainTheorem}

Throughout this section, $A$ denotes a finite-dimensional semisimple $\mathbb{Q}$-algebra and $\mathfrak{A}\subset A$ is a $\mathbb{Z}$-order. The aim of this section is to show Theorem \ref{theorem:MainTheorem} from the introduction.  We proceed in four steps.

\subsection{\texorpdfstring{The localization $\QC\colon \LCAA \to \LCAA / \LCAAC$}{First quotient}}
The following proposition (see \cite{HenrardVanRoosmalen19b}) reduces the study of localizing invariants of $\LCAA$ (such as non-connective $K$-theory) to that of the quotient category $\LCAA/\LCAAC$, which we shall call $\EE$.

\begin{proposition}\label{proposition:FirstLocalization}
	The subcategory $\LCAAC\subseteq\LCAA$ is a strictly inflation-percolating subcategory.  The quotient $\QC\colon \LCAA\to \EE (=\LCAA / \LCAAC)$ induces an exact sequence of stable $\infty$-categories
	\[\Dbinf(\LCAAC)\to \Dbinf(\LCAA)\to \Dbinf(\EE).\]
	As every object in $\Dbinf(\LCAAC)$ can be trivialized using an Eilenberg swindle with infinite products, for any localizing invariant $K$, there is an equivalence $K(\LCAA) \simeq K(\EE).$
\end{proposition}

\subsection{\texorpdfstring{The functor $\QV\colon \LCAA / \LCAAC \to \FF$}{Second quotient}}  We now write $\VV$ for the full additive subcategory of $\EE$ generated by the vector $\mathfrak{A}$-modules. Our first goal is to show that $\VV$ is a strictly inflation-percolating subcategory of $\ex{\EE}$, the exact hull of $\EE$, so that we can consider the quotient $\FF \coloneqq \ex{\EE}/\VV.$  We start with the following lemma.

\begin{lemma}
\makeatletter
\hyper@anchor{\@currentHref}%
\makeatother\label{lemma:BasicLemma}
	\begin{enumerate}
		\item\label{item:NaturalEquivalence} For any vector $\mathfrak{A}$-module $V$, the localization functor $\QC$ induces a natural equivalence
		\[\QC\colon \Hom_{\LCAA}(-,V)\to \Hom_{\EE}(\QC(-),\QC(V) ).\]
		In particular, it follows that $V$ is injective in $\EE$.
		\item\label{item:VVisEquivalentToModAR} The category $\VV$ is equivalent to the category $\smod(A_{\mathbb{R}})$ where $A_{\mathbb{R}}\coloneqq\mathfrak{A}\otimes_{\mathbb{Z}}\mathbb{R}$. In particular, $\VV$ is a fully exact abelian subcategory of $\EE$.
	\end{enumerate}
\end{lemma}

\begin{proof}
	\begin{enumerate}
		\item Let $f\in \Hom_{\LCAA}(X,V)$ such that $\QC(f)=0$. There exists a weak isomorphism $t\colon V \stackrel{\sim}{\rightarrow} Y$ in $S_{\LCAAC}$ (thus, with $\ker f \in \LCAAC$) such that $t\circ f=0$ in $\LCAA$.  As the only compact submodule of $V$ is trivial, $t$ is a monomorphism.  It follows that $f=0$ in $\LCAA$. This shows $\Hom_{\LCAA}(X,V)\to \Hom_{\EE}(X,V)$ is an injection.
		
		To show that it is a surjection, let $g\in \Hom_{\EE}(X,V)$ be represented by a roof $X\stackrel{f}{\rightarrow}Y \stackrel{s}{\leftarrow} V$ with $s\in S_{\LCAAC}$. Note that $s$ is an inflation (as $s$ is strict and $V$ only has the trivial compact submodule).  As $V$ is injective in $\LCAA$, the inflation $V\stackrel{s}{\inflation}Y$ is a coretraction. Let $t\colon Y\deflation V$ be the corresponding retraction, i.e.~$ts=1_V$. It follows that $\QC(t\circ f)=g$. This shows the desired bijection.
		
As $V$ is injective in $\LCAA$, $\Hom_{\LCAA}(-,V)$ is an exact functor. Hence $\Hom_{\EE}(-,V)$ is exact as well and thus $V$ is injective in $\EE$ (this characterization of injective objects remains valid for inflation-exact categories, see \cite[Proposition~3.22]{HenrardVanRoosmalen19b}).
		\item This follows immediately from the above equivalence.\qedhere
	\end{enumerate}
\end{proof}

\begin{proposition}\label{proposition:VVSatisfiesA1AndA2InEE}
	The category $\VV$ is a strictly inflation-percolating subcategory of the exact hull $\ex{\EE}$ of $\EE$.
\end{proposition}

\begin{proof}
	By Proposition \ref{proposition:A1A2+InjectiveLiftsToHull} and Lemma \ref{lemma:BasicLemma}.(\ref{item:NaturalEquivalence}), it suffices to show that $\VV\subseteq \EE$ satisfies axioms \ref{A1} and \ref{A2}.

	We first show axiom \ref{A2}. Let $f\in \Hom_{\EE}(V,X)$ be represented by a roof $V \stackrel{g}{\rightarrow} Y \stackrel{s}{\leftarrow}X$ with $s\in S_{\LCAAC}$. As the image of a connected space is connected, the Structure Theorem yields the following commutative diagram in $\LCAA$:
		\[\xymatrix{
			& D_Y & \\
			V\ar[r]^g\ar@{.>}[ru]^0\ar@{.>}[rd]_h & Y\ar@{->>}[u]^{p_Y} & X\ar[l]^{s}\\
			& C_Y \oplus V_Y\ar@{>->}[u]^{i_Y} & 
		}\] Here $V_Y$ is a vector $\mathfrak{A}$-module. As the projection $C_Y \oplus V_Y \deflation V_Y$ is an isomorphism in $\EE$ and the composition $V \to V_Y$ is strict, we know that $V\to C_Y\oplus V_Y$ is strict in $\EE$.  As $\EE$ is an inflation-exact category, axiom \ref{L1} yields that the composition of inflations is an inflation. It follows that the composition $i_Y\circ h$ is strict in $\EE$. This shows axiom \ref{A2}.
	
	We now show axiom \ref{A1}, i.e.~that $\VV$ is a Serre subcategory of $\EE$ . Let $X\stackrel{\iota}{\inflation} Y\stackrel{\rho}{\deflation} Z$ be a conflation in $\EE$. Assume that $Y\in \VV$. By axiom $\ref{A2}$, $\rho$ is strict with image in $\VV$, thus $Z\in \VV$. As $X\cong \ker(\rho)$ is the kernel of a morphism in $\VV$ and $\VV\subseteq\EE$ is an abelian subcategory by Lemma \ref{lemma:BasicLemma}.\eqref{item:VVisEquivalentToModAR}, $X\in \VV$. 
	
	Conversely, assume that $X,Z\in \VV$. By Lemma \ref{lemma:BasicLemma}.\eqref{item:NaturalEquivalence}, $X$ is injective in $\EE$. It follows that the conflation $(\iota,\rho)$ splits and thus $Y\cong X\oplus Z$ belongs to $\VV$.
\end{proof}

\begin{corollary}\label{Corollary:SecondFibrationSequence}
The quotient functor $\QV\colon \ex{\EE}\to \FF (\coloneqq \ex{\EE} / \VV)$ induces a fibre sequence
	\[K(\VV)\to K(\ex{\EE})\to K(\FF),\]
where $K$ is any localizing invariant.
\end{corollary}

\begin{proof}
By Lemma \ref{lemma:BasicLemma}, we know that $\VV$ contains enough $\ex{\EE}$-injective objects.  The statement then follows from Theorem \ref{theorem:Quotient/LocalizationTheorem}.
\end{proof}

\begin{proposition}\label{proposition:UniversalQCV}
The functor $\QCV\colon \LCAA \to \EE \to \ex{\EE} \to \FF$ is 2-universal with respect to the conflation-exact functors $F\colon \LCAA \to \CC$ with $\CC$ exact and $F(\LCAACV) = 0$, thus the functor $-\circ \QCV\colon \Fun(\FF, \CC) \to \Fun(\LCAA, \CC)$ is a fully faithful functor whose essential image consists of those $F\colon \LCAA \to \CC$ for which $F(\LCAACV) = 0$.
\end{proposition}

\begin{proof}
From combining each of the universal properties of $\LCAA \to \EE \to \ex{\EE} \to \FF$.
\end{proof}

\begin{remark}\label{remark:QCVFinitelyGenerated}
Note that in $\EE$, we have $\A \cong \bR \otimes_\bZ \A \in \VV$.  In particular, $\QCV(\A) = 0.$  Hence, $\QCV$ sends every finitely generated discrete $\A$-module to zero. 
\end{remark}

\subsection{\texorpdfstring{The equivalence $\FF \simeq \Mod \A / \mod \A$}{An equivalence}}

In order to complete the proof of Theorem \ref{theorem:MainTheorem}, we show that $\FF \simeq \Mod \A / \mod \A$ (see Proposition \ref{proposition:PsiPhi}).  For this, consider the localization functors $\QA\colon \Mod \A \to \Mod \A / \mod \A$ and $\QfA\colon \Mod \A \to \Mod \A / \fmod \A$, where $\fmod \A$ is the full subcategory of $\Mod \A$ consisting of finite $\A$-modules.  It follows from Proposition \ref{proposition:fLocalization}.(\ref{enumerate:fLocalization4}) that $\Qf(\LCAAD) \simeq \Mod \A / \fmod \A.$  Moreover, the universal property of $\QfA$ shows that there is a unique functor $\QA'\colon \Mod \A / \fmod \A \to \Mod \A / \mod \A$ such that $\QA = \QA' \circ \QfA$.

The torsion-free part functor $D\colon \oLCAA\to \oLCAAD\colon M \mapsto D_M$ from Theorem \ref{theorem:TorsionPair} need not be conflation-exact.  This can be seen by setting $\A = \bZ$ and starting from the conflation $\bZ \inflation \bR \deflation \bR / \bZ.$  However, we need not change much to obtain a conflation-exact functor.

\begin{proposition}\label{proposition:ConflationExact}
The functor $\QA' \circ D \colon \LCAA / \LCAAf \to \Mod \A/ \mod \A$ is conflation-exact.
\end{proposition}

\begin{proof}
Let $X \inflation Y \deflation Z$ be a conflation in $\LCAA / \LCAAf.$  The Structure Theorem of $\oLCAA$ gives the following commutative diagram
\[\xymatrix{
C_X\ar@{>->}[r] \ar@{>->}[d] & X\ar@{->>}[r] \ar@{>->}[d] & {D_X} \ar[d]^{g} \\
C_Y\ar@{>->}[r] & Y\ar@{->>}[r] & {D_Y}}\]
where the left vertical arrow is an inflation by the dual of \cite[Proposition 7.6]{Buhler10} and the rightmost vertical arrow is strict by Proposition \ref{proposition:fLocalization}.(\ref{enumerate:fLocalization5}).  Applying the Short Snake Lemma (\cite[Corollary 8.13]{Buhler10}), we find exact sequences $\ker g \stackrel{i}{\inflation} C_Y /C_X \to Z \deflation \coker g$ and $\ker g \inflation D_X \stackrel{g}{\rightarrow} D_Y \deflation \coker g.$

It follows from Corollary \ref{corollary:DiscreteInCV} that $C_Y / C_X \in \LCAACV$ and hence $\coker i \in \LCAACV.$  Likewise, we find that $\coker g \in \LCAAD.$  This shows that the conflation $\coker i \inflation Z \deflation \coker g$ is the torsion / torsion-free conflation of $Z$ from Theorem \ref{theorem:TorsionPair}, hence $\coker g \cong D_Z.$

Moreover, it follows from Corollary \ref{corollary:DiscreteInCV} that $\ker g$ is finitely generated and discrete.  Hence, we find a conflation $\QA'(D_X) \inflation \QA'(D_Y) \deflation \QA'(D_Z)$ in $\Mod \A / \mod \A,$ as required.
\end{proof}

\begin{figure}[tb]
	\[\xymatrix{
		{\QA\colon} &{\Mod \A} \ar[rr]^-{\QfA} \ar[d] && {\Mod \A / \fmod \A} \ar[rr]^{\QA'} \ar[d]^{R} && {\Mod \A / \mod \A} \ar[d]^{\Phi} & \\
		{\QCV\colon} &{\LCAA} \ar[rr]^-{\Qf} && {\LCAA / \LCAAf} \ar[rr]^{Q'} \ar[d]^{D} && {\FF} \ar[d]^{\Psi}  &\\
		&&& {\Mod \A / \fmod \A} \ar[rr]^{\QA'} && {\Mod \A / \mod \A}	}\]
	\caption{Overview of the functors from Construction \ref{construction:Functors}.}
	\label{figure:Functors}
\end{figure}

\begin{construction}\label{construction:Functors}
We now construct the diagram given in Figure \ref{figure:Functors}.  We start with the rows.  The functor $\QA'$ is the unique functor such that $\QA = \QA'\circ \QfA,$ and $Q'$ is the unique functor such that $\QCV = Q' \circ \Qf$; these are induced by the universal properties of $\QfA$ and $\Qf$, respectively.

For the columns, the functor $\Mod \A \to \LCAA$ is the functor mapping an $\A$-module to the corresponding discrete $\A$-module.  The functor $R$ is the unique functor making the top-left square commute (it exists by the universal property of $\QfA$).  The (essential) image of $R$ corresponds to the torsion-free part of the torsion pair in Theorem \ref{theorem:TorsionPair}, hence $R$ has a left adjoint $D\colon \LCAA / \LCAAf \to \Mod \A / \fmod \A,$ given by mapping any object $X \in \LCAA / \LCAAf$ to its torsion-free part $D_X.$  By construction, $D \circ R \cong 1.$

In the last column, the functor $\Phi\colon \Mod \A / \mod \A \to \FF$ is the unique functor making the top rectangle commute; it exists by the universal property of $\QA\colon \Mod \A \to \Mod \A / \mod \A$ (see Remark \ref{remark:QCVFinitelyGenerated}).  Note that $\Phi$ is also the unique functor such that $\Phi \circ \QA' = Q' \circ R.$

The functor $\Psi\colon \FF \to \Mod \A / \mod \A$ is a functor such that $\Psi \circ \QCV \cong \QA' \circ D \circ \Qf;$ it exists by the universal property of $\QCV$ (see Proposition \ref{proposition:UniversalQCV}, note that $\QA' \circ D \circ \Qf$ is conflation-exact by Proposition \ref{proposition:ConflationExact}).
\end{construction}

\begin{proposition}\label{proposition:PsiPhi}
The functors $\Psi$ and $\Phi$ are quasi-inverses.
\end{proposition}

\begin{proof}
For each $M \in \LCA$, the map $p_M\colon M \mapsto R(D_M)$ corresponds to the unit of the adjunction $D \dashv R$.  As $Q'(p_M)$ is an isomorphism, we find that $\Phi \circ \Psi \circ \QCV \cong Q' \circ R \circ D \circ \Qf$ is isomorphic to $\QCV = Q' \circ \Qf.$  It follows from the universal property of $\QCV$ that $1_\FF \to \Phi \circ \Psi$ is a natural equivalence.

For the other direction, we start from $\QA' \cong \QA'\circ D \circ R = \Psi \circ \Phi \circ \QA',$ so that the universal property of $\QA'$ yields that $\Psi \circ \Phi \cong 1$, as required. 
\end{proof}

\subsection{\texorpdfstring{Proof of Theorem \ref{theorem:MainTheorem}}{Proof of the main Theorem}}  We are now in a position to prove the main theorem.

\begin{proof}[Proof of Theorem \ref{theorem:MainTheorem}]
	Consider the essentially commutative diagram
	\[\xymatrix{
		\smod(\mathfrak{A})\ar[r]\ar[d] & \Mod(\mathfrak{A})\ar[r]^-{Q_{\mathfrak{A}}}\ar[d] & \Mod(\mathfrak{A})/\smod(\mathfrak{A})\ar[d]^{\rotatebox{90}{$\simeq$}}_{\Phi}\\
		\VV\ar[r] & {\ex{\EE}}\ar[r] & \FF
	}\] of functors, lifting to an essentially commutative diagram of the bounded derived $\infty$-categories (where the rows are exact sequences).  It was shown in Lemma \ref{lemma:BasicLemma}.(\ref{item:VVisEquivalentToModAR}) that $\VV\simeq \mod \A_\bR$; the leftmost downwards arrow is given by $M \mapsto \bR \otimes_\bZ M$.  This induces a bicartesian square of stable $\infty$-categories
	\[\xymatrix{
	\Dbinf(\mod \A) \ar[r] \ar[d] & \Dbinf(\Mod \A) \ar[d] \\
	\Dbinf(\mod \A_\bR) \ar[r] & \Dbinf(\ex{\EE}).}\]
	Using the Eilenberg swindle with direct sums, shows that every object in $\Dbinf(\Mod \A)$ gets trivialized under a localizing invariant $K\colon \StabInfCat \to \AA.$  Hence, for each such $K$, there is a fiber sequence $K(\Dbinf(\mod \A)) \to K(\Dbinf(\mod \A_\bR)) \to K(\Dbinf(\ex{\EE})).$  Combining Theorem \ref{proposition:ExactHull} and Proposition \ref{proposition:FirstLocalization}, we find that $K(\Dbinf(\LCAA)) \simeq K(\Dbinf(\EE)) \simeq K(\Dbinf(\ex{\EE})).$ Using our convention to suppress $\Dbinf$ in the notation whenever convenient, this yields the required fiber sequence as formulated in the introduction.
\end{proof}

\providecommand{\bysame}{\leavevmode\hbox to3em{\hrulefill}\thinspace}
\providecommand{\MR}{\relax\ifhmode\unskip\space\fi MR }
\providecommand{\MRhref}[2]{%
  \href{http://www.ams.org/mathscinet-getitem?mr=#1}{#2}
}
\providecommand{\href}[2]{#2}


\begin{thebibliography}{10}

\bibitem{BazzoniCrivei13}
Silvana Bazzoni and Septimiu Crivei, \emph{One-sided exact categories}, J. Pure
  Appl. Algebra \textbf{217} (2013), no.~2, 377--391.

\bibitem{MR3070515}
Andrew~J. Blumberg, David Gepner, and Gon\c{c}alo Tabuada, \emph{A universal
  characterization of higher algebraic {$K$}-theory}, Geom. Topol. \textbf{17}
  (2013), no.~2, 733--838.

\bibitem{Braunling18}
Oliver Braunling, \emph{On the relative {$K$}-group in the {ETNC}}, New York J.
  Math. \textbf{25} (2019), 1112--1177.

\bibitem{Buhler10}
Theo B\"{u}hler, \emph{Exact categories}, Expo. Math. \textbf{28} (2010),
  no.~1, 1--69.

\bibitem{clausen}
Dustin Clausen, \emph{A {K}-theoretic approach to {A}rtin maps},
  arXiv:1703.07842 [math.KT] (2017).

\bibitem{HenrardVanRoosmalen19b}
Ruben Henrard and Adam-Christiaan van Roosmalen, \emph{Derived categories of
  (one-sided) exact categories and their localizations}, arXiv preprint
  arXiv:1903.12647 (2019).

\bibitem{HenrardVanRoosmalen19a}
\bysame, \emph{Localizations of (one-sided) exact categories}, arXiv preprint
  arXiv:1903.10861 (2019).

\bibitem{MR2329311}
Norbert Hoffmann and Markus Spitzweck, \emph{Homological algebra with locally
  compact abelian groups}, Adv. Math. \textbf{212} (2007), no.~2, 504--524.
  \MR{2329311 (2009d:22006)}

\bibitem{Keller90}
Bernhard Keller, \emph{Chain complexes and stable categories}, Manuscripta
  Math. \textbf{67} (1990), no.~4, 379--417.

\bibitem{Morris77}
Sidney~A. Morris, \emph{Pontryagin duality and the structure of locally compact
  abelian groups}, vol.~29, Cambridge University Press, 1977.

\bibitem{Moskowitz67}
Martin Moskowitz, \emph{Homological algebra in locally compact abelian groups},
  Transactions of the American Mathematical Society \textbf{127} (1967), no.~3,
  361--404.

\bibitem{Rump11}
Wolfgang Rump, \emph{On the maximal exact structure on an additive category},
  Fund. Math. \textbf{214} (2011), no.~1, 77--87.

\bibitem{Schlichting04}
Marco Schlichting, \emph{Delooping the {$K$}-theory of exact categories},
  Topology \textbf{43} (2004), no.~5, 1089--1103.

\end{thebibliography}
\end{document}